\documentclass{amsart}
\usepackage{amsmath}
\usepackage{amssymb}
\usepackage{fancybox}
\usepackage{eucal}
\usepackage{amsthm}
\usepackage{amscd}
\usepackage{hyperref}
\usepackage{wasysym}
\usepackage{url}
\usepackage{mathrsfs}
\usepackage{amsfonts}

\usepackage{amssymb,}
\usepackage[]{amsmath, amsthm, amsfonts,graphicx, amscd,}
\usepackage[all,cmtip]{xy}
\input amssym.def 
\input amssym
\DeclareMathOperator*{\forkindep}{\raise0.2ex\hbox{\ooalign{\hidewidth$\vert$\hidewidth\cr\raise-0.9ex\hbox{$\smile$}}}}

\makeatletter
\let\@wraptoccontribs\wraptoccontribs
\makeatother

\DeclareMathOperator{\ord}{ord}

\begin{document}

\newtheorem {thm}{Theorem}[section]
\newtheorem{theorem}{Theorem}
\newtheorem{lemma}[theorem]{Lemma}
\newtheorem{claim}[theorem]{Claim}
\newtheorem{definition}[theorem]{Definition}
\newtheorem{corollary}[theorem]{Corollary}
\newtheorem{aside}[theorem]{Aside}
\newtheorem{corr}[thm]{Corollary}
\newtheorem{fact}[thm]{Fact}
\newtheorem {cl}[thm]{Claim}
\newtheorem*{thmstar}{Theorem}
\newtheorem{prop}[thm]{Proposition}
\newtheorem{proposition}[thm]{Proposition}
\newtheorem*{propstar}{Proposition}
\newtheorem {lem}[thm]{Lemma}
\newtheorem*{lemstar}{Lemma}
\newtheorem{conj}[thm]{Conjecture}
\newtheorem{question}[thm]{Question}
\newtheorem*{questar}{Question}
\newtheorem{ques}[thm]{Question}
\newtheorem*{conjstar}{Conjecture}
\theoremstyle{remark}
\newtheorem{remark}[thm]{Remark}
\newtheorem{rem}[thm]{Remark}
\newtheorem{np*}{Non-Proof}
\newtheorem*{remstar}{Remark}
\theoremstyle{definition}
\newtheorem{defn}[thm]{Definition}
\newtheorem*{defnstar}{Definition}
\newtheorem{exam}[thm]{Example}
\newtheorem{example}[theorem]{Example}
\newtheorem*{examstar}{Example}
\newcommand{\pd}[2]{\frac{\partial #1}{\partial #2}}
\newcommand{\pp}{\partial }
\newcommand{\pdtwo}[2]{\frac{\partial^2 #1}{\partial #2^2}}
\def\Ind{\setbox0=\hbox{$x$}\kern\wd0\hbox to 0pt{\hss$\mid$\hss} \lower.9\ht0\hbox to 0pt{\hss$\smile$\hss}\kern\wd0}
\def\Notind{\setbox0=\hbox{$x$}\kern\wd0\hbox to 0pt{\mathchardef \nn=12854\hss$\nn$\kern1.4\wd0\hss}\hbox to 0pt{\hss$\mid$\hss}\lower.9\ht0 \hbox to 0pt{\hss$\smile$\hss}\kern\wd0}
\def\ind{\mathop{\mathpalette\Ind{}}}
\def\nind{\mathop{\mathpalette\Notind{}}} 
\newcommand{\m}{\mathbb }
\newcommand{\mc}{\mathcal }
\newcommand{\mf}{\mathfrak }
\newcommand{\is}{^{p^ {-\infty}}}
\newcommand{\codim}{\operatorname{codim}}
\newcommand{\Gal}{\operatorname{Gal}}
\newcommand{\Num}{\operatorname{Num}}
\newcommand{\Cl}{\operatorname{Cl}}
\newcommand{\Div}{\operatorname{Div}}
\newcommand{\Ann}{\operatorname{Ann}}
\newcommand{\Frac}{\operatorname{Frac}}
\newcommand{\lcm}{\operatorname{lcm}}
\newcommand{\height}{\operatorname{ht}}
\newcommand{\Der}{\operatorname{Der}}
\newcommand{\Pic}{\operatorname{Pic}}
\newcommand{\Sym}{\operatorname{Sym}}
\newcommand{\Proj}{\operatorname{Proj}}
\newcommand{\characteristic}{\operatorname{char}}
\newcommand{\Spec}{\operatorname{Spec}}
\newcommand{\Hom}{\operatorname{Hom}}
\newcommand{\res}{\operatorname{res}}
\newcommand{\Aut}{\operatorname{Aut}}
\newcommand{\length}{\operatorname{length}}
\newcommand{\Log}{\operatorname{Log}}
\newcommand{\Set}{\operatorname{Set}}
\newcommand{\Fun}{\operatorname{Fun}}
\newcommand{\id}{\operatorname{id}}
\newcommand{\Gp}{\operatorname{Gp}}
\newcommand{\Ring}{\operatorname{Ring}}
\newcommand{\Mod}{\operatorname{Mod}}
\newcommand{\Mor}{\operatorname{Mor}}
\newcommand{\SheafHom}{\mathcal{H}om}
\newcommand{\pre}{\operatorname{pre}}
\newcommand{\coker}{\operatorname{coker}}
\newcommand{\acl}{\operatorname{acl}}
\newcommand{\dcl}{\operatorname{dcl}}
\newcommand{\tp}{\operatorname{tp}}
\newcommand{\dom}{\operatorname{dom}}
\newcommand{\val}{\operatorname{val}}
\newcommand{\Aa}{\mathbb{A}}
\newcommand{\Qq}{\mathbb{Q}}
\newcommand{\Rr}{\mathbb{R}}
\newcommand{\Zz}{\mathbb{Z}}
\newcommand{\Nn}{\mathbb{N}}
\newcommand{\Cc}{\mathbb{C}}
\newcommand{\Ii}{\mathbb{I}}
\newcommand{\Gg}{\mathbb{G}}
\newcommand{\Uu}{\mathbb{U}}
\newcommand{\Mm}{\mathbb{M}}
\newcommand{\Ff}{\mathbb{F}}
\newcommand{\Pp}{\mathbb{P}}

\def\bu{{\rm{u}}}
\def\bv{{\rm{v}}}
\def\V{\mathbb{V}}
\def\P{\mathbb{P}}
\def\L{\mathbb{L}}
\def\Y{\mathbb{Y}}
\def\I{\mathbb{I}}
\def\den{{\rm{den}}}
\def\num{{\rm{num}}}
\def\card{{\rm{card}}}
\def\CI{{\mathcal{I}}}
\def\ld{{\rm{ld}}}
\def\init{{\rm{init}}}
\def\Sep{\rm {S}}
\def\chow{{\rm{Chow}}}
\def\H{{\rm{H}}}
\def\denom{{\rm{denom}}}
\def\det{{\rm{det}}}
\def\sat{{\rm{sat}}}
\def\asat{{\rm{asat}}}
\def\ff{\mathcal{F}}
\def\VB{{\mathbf{V}}}
\def\CV{{\mathbb{CV}}}
\def\pr{{\rm{pr}}}
\def\id{{\rm{id}}}
\def\lv{{\rm{lv}}}
\def\dim{{\rm{dim}}}
\def\trdeg{\hbox{\rm{tr.deg}}}
\def\dtrdeg{\hbox{\rm{d.tr.deg}}}
\def\ZZ{\mathbb Z}
\def\NN{\mathbb N}
\def\d{\delta}
\def\D{\Delta}
\def\ord{\operatorname{ord}}

\title[Definability of Kolchin polynomials]{Effective definability of Kolchin polynomials} 

\author{James Freitag}
\address{James Freitag\\
University of Illinois Chicago\\
Department of Mathematics, Statistics, and Computer Science\\
851 S. Morgan Street\\
Chicago, IL, USA, 60607-7045.}
\email{freitag@math.uic.edu}

\author{Omar Le\'on S\'anchez}
\address{Omar Le\'on S\'anchez\\
University of Manchester\\
School of Mathematics\\
Oxford Road \\
Manchester, UK, M13 9PL.}
\email{omar.sanchez@manchester.ac.uk}

\author{Wei Li}
\address{Wei Li\\
Academy of Mathematics and Systems Science\\
Chinese Academy of Sciences\\
No.55 Zhongguancun East Road\\
Beijing, China, 100190.}
\email{liwei@mmrc.iss.ac.cn}

\date{\today}
\subjclass[2010]{12H05, 14Q20}
\keywords{differential fields, Kolchin polynomial, effective definability in families}
\thanks{JF is partially supported by  NSF Grant 1700095. WL is partially supported by  NSFC Grants (11688101, 11301519, 11671014)}

\begin{abstract}
While the natural model-theoretic ranks available in differentially closed fields (of characteristic zero), namely Lascar and Morley rank, are known \emph{not} to be definable in families of differential varieties; in this note we show that the differential-algebraic rank given by the Kolchin polynomial is in fact definable. As a byproduct, we are able to prove that the property of being \emph{weakly irreducible} for a differential variety is also definable in families. The question of full irreducibility remains open, it is known to be equivalent to the generalized Ritt problem.
\end{abstract}

\maketitle

\tableofcontents

\section{Introduction and some preliminaries} \label{intro}

Fix a differentially closed field of characteristic zero $(K, \Delta)$ with $\Delta =\{\delta_1,\dots,\delta_m\}$ a set of distinguished commuting derivations. We let $k$ be a differential subfield of $K$. Furthermore, we assume that $K$ is ``big''; in the sense that it is a universal model for differential-algebraic geometry. In particular, $K$ is universal over $k$ (or over any ``small" differential subfield for that matter). In model-theoretic parlance, $K$ is a saturated model of the theory $\operatorname{DCF}_{0,m}$. 

%; and let $\Theta$ be the set of derivative operators $\{\delta\xi:\xi\in \NN^m\}$. For a fixed $s\geq 0$, we let $\Theta_s$ be $\{\d^\xi\in \Theta: \ord\xi\leq s\}$.

%Let $x=(x_1,\dots,x_n)$ be an $n$-tuple of differential indeterminates for a fixed $n\geq 1$. We let $\Theta(x)$ be the set of algebraic indeterminates $\{\d^\xi x_i:\d^\xi\in \Theta,i=1,\dots,n\}$, and $\Theta_s(x)$ those algebraic indeterminates of order at most $s$. 

Recall that a numerical polynomial (in one variable) is a polynomial $p\in \mathbb{Q}[t]$ such that $p(s)\in \ZZ$ for all integers $s$. Numerical polynomials always have the form
$$p(t)=\sum_{i=0}^d a_i\binom{t+i}{i}$$
for some integers $a_i$'s. The tuple $(a_d,\dots,a_0)$ is usually called the standard coefficients of $p$.

For an $m$-tuple $\xi=(u_1,\dots,u_m)\in \NN^m$, we let $\ord(\xi)=u_1+\cdots +u_m$, and we use multi-index notation to denote derivative operators; that is, $\d^\xi=\d_1^{u_1}\cdots\d_m^{u_m}$. Let us recall a classical result of Kolchin:

\begin{fact}[\cite{KolchinDAAG}, Chap.2, \S 12]\label{kolpol}
Let $a=(a_1,\dots,a_n)$ be a tuple from $K$. There exists a numerical polynomial $\omega_{a/k}$ with the following properties:
\begin{enumerate}
\item [(i)] For sufficiently large $s\in \NN$, $\omega_{a/k}(s)$ equals the transcendence degree of $k(\d^\xi a_i:\ord\xi\leq s, i=1,\dots,n)$ over $k$.
\item [(ii)]  $\deg \omega_{a/k}\leq m$.
\item [(iii)] If we write $\omega_{a/k}(t)=\sum_{i=0}^m a_i\binom{t+i}{i}$ where $a_i\in \ZZ$, then $a_m$ equals the differential-transcendence degree of the differential field $k\langle a\rangle$.
\end{enumerate}
\end{fact}

The polynomial $\omega_{a/k}$ is called the \emph{Kolchin polynomial of} $a$ \emph{over} $k$ (or the differential dimension polynomial of $a$ over $k$). Even though it is not generally a differential-birational invariant of $a$, it serves as an important measure of the transcendentality of $a$. For instance, if $k\subseteq L\subseteq  K$ are differential fields and $a\in K$, then $k\langle a\rangle$ is algebraically disjoint from $L$ over $k$ iff $\omega_{a/k}=\omega_{a/L}$. 

Let $x=(x_1,\dots,x_n)$ be an $n$-tuple of differential indeterminates for a fixed $n\geq 1$. If $P$ is a prime differential ideal of the differential polynomial ring $k\{x\}$, we define the Kolchin polynomial of $P$, $\omega_{P}$, as the Kolchin polynomial of a generic point $a\in K^n$ of $P$ over $k$. Similarly, if $V$ is an irreducible differential variety over $k$, we define the Kolchin polynomial of $V$, $\omega_{V}$, as the Kolchin polynomial of the prime differential ideal in $k\{x\}$ given by the vanishing of $V$. For an arbitrary differential variety (not necessarily irreducible) over $k$, the Kolchin polynomial of $V$ is 
$$\omega_V:=\max_{\leq}\left\{ \omega_W: W\textrm{ is an irreducible component of } V\right\}$$
where $\leq$ denotes the total ordering on the set of numerical polynomials by eventual domination; i.e., $p\leq q$ if and only if $p(s)\leq q(s)$ for all sufficiently large $s\in \NN$ (equivalently, the standard coefficients of $p$ are less than or equal to those of $q$ in the lexicographical order).

\begin{remark}\label{wellorder}
The Kolchin polynomial has the following two important properties (for more properties see Chapter II of \cite{KolchinDAAG}):
\begin{enumerate}
\item If $V\subseteq W$ are irreducible differential varieties over $k$ with generic points $a$ and $b$, respectively, then $\omega_{a/k}\leq \omega_{b/k}$ with equality if and only $V=W$.
\item The collection of Kolchin polynomials is well-ordered by eventual domination (see \cite{SitWell}).
\end{enumerate}
\end{remark}

\smallskip

Let $F(x,y)$ be a collection of differential polynomials over $k$, where $x$ and $y$ are tuples of differential indeterminates (not necessarily of the same length). Note that for each $a\in K^{|y|}$, the system $F(x,a)=0$ defines a differential variety $V_a$ over $k\langle a\rangle$. Any such collection of $V_a$'s will be called a (definable) family of differential varieties, we denote this by $(V_a)$. We say that the family has order $r$ if the differential polynomials in $F(x,y)$ have order at most $r$ in the variable $x$ and $r$ is minimal such (and similarly for the degree of the family). 

Given a definable family of differential varieties $(V_a)$ and a numerical polynomial $p$, the questions that drive the results in this paper are the following: 

\begin{question}
Is the set 
\begin{equation}\label{definable}
\{a \, : \, \omega _{V_a }= p \} 
\end{equation}
definable in the structure $(K,\D)$? And, can one prove this effectively (i.e., effectively produce a formula defining \eqref{definable})? 
\end{question}

We answer both questions affirmatively. The method of our proof uses recently established bounds for the order of characteristic sets of prime differential ideals to secure an effective value $s_0$ (which depends only on the order of the family) such that for $s\geq s_0$ we have that $\omega_{V_a}(s)$ equals 
\begin{equation}\label{trdeg}
\operatorname{trdeg}_k k(\d^{\xi}b_i:\ord\xi\leq s, i=1,\dots,n)
\end{equation}
where $b=(b_1,\dots,b_n)$ is a generic point of any irreducible component of $V_a$ of maximal Kolchin polynomial (among all components). We do this in Section~\ref{charpoly}. We then prove in Section~\ref{effective}, using classical algebro-geometric facts and effective results in the theory of prolongation spaces, that one can effectively determine those $a$ for which \eqref{trdeg} equals a fixed nonnegative integer for each natural number $s\geq s_0$. The main result follows more or less immediately from this (see Theorem~\ref{Boundy}).

Finally, in Section~\ref{applications}, we apply the definability result to show that the property of being weakly irreducible (meaning that there is only one component of maximal Kolchin polynomial) is too definable in families. We will also see that any given definable family $(V_a)$ admits only finitely many Kolchin polynomials; that is, the set $\{\omega_{V_a}: \text{as } a \text{ varies}\}$ is finite.

\begin{rem} 
Besides Kolchin polynomials, various other (model-theoretic) ranks have been studied in differentially closed fields, see for instance \cite{PongRank}. Some of these ranks are known \emph{not} to be definable in families. For instance, working in the ordinary case $m=1$, from work of the Japanese school of integrable systems and the trichotomy theorem for $DCF_0$, Pillay and Nagloo show that Morley rank (and also Lascar rank and the differential version of Krull dimension) are not definable in families \cite[Corollary 3.5]{nagloo2011algebraic}. Specifically, for $\alpha \in \m C$, the solution set $P_{II} (\alpha)$ to the second Painlev\'e equation is strongly minimal if and only if $\alpha \in \frac{1}{2} + \m Z$. So, in the family of differential varieties 
$$\{ P_{II} (\alpha) \, | \, \alpha\in K \},$$
the collection of those $\alpha$ such that $P_{II} (\alpha)$ has Morley rank one (in this case the Morley rank of each fibre is equal to its Lascar rank and differential Krull dimension) is not definable. 
\end{rem}

\section{On characteristic sets and numerical polynomials}\label{charpoly}

We carry forward the notation from the previous section. In particular, $k$ denotes a (small) differential subfield of our universal differentially closed field $(K,\D)$. 

As we pointed out in Fact \ref{kolpol}, for all large enough values of $s$, the Kolchin polynomial of an irreducible differential variety $V$ over $k$ is given by a transcendence degree calculation.
The minimum  $i_0\in\mathbb N$ such that 
$$\omega_V(s)=\operatorname{trdeg}_k k(\d^{\xi}b_i:\ord\xi\leq s, 1\leq i\leq n)$$ 
for all $s\geq i_0$, is known as the {\em Hilbert-Kolchin regularity} of $V$, where $(b_1,\ldots,b_n)$ is a generic point of $V$. Upper bounds of this regularity number were estimated for  quasi-regular ordinary differential systems \cite{DGMS}. In this section (see Theorem~\ref{sbound} below), we effectively compute an upper bound  on the Hilbert-Kolchin regularity for  irreducible components of differential  varieties depending only on $m$, $n$ and the maximal order of the system.

We need to recall the notion of volume for subsets of $\NN^m$. We let $\leq$ denote the product order on $\NN^m$; that is, $(u_1,\dots,u_m)\leq (v_1,\dots,v_m)$ means that $u_i\leq v_i$ for $i=1,\dots,m$. Given any $E\subseteq \NN^m$ and a nonnegative integer $s$, the \emph{volume of $E$ at level $s$} is
$$V_E(s)=\{\xi\in \NN^m: \ord\xi\leq s \text{ and } \xi\not\geq \eta \text{ for all }\eta\in E\}.$$
In \cite[Chapter 0, \S17]{KolchinDAAG}, Kolchin shows that for any $E\subseteq \NN^m$ there is a numerical polynomials $\omega_E(t)$ such that for sufficiently large $s\in \NN$
\begin{equation}\label{equalcomp}
\omega_E(s)=|V_E(s)|.
\end{equation}
Furthermore, $\deg \omega_E\leq m$; equality occurs if and only if $E$ is empty, in which case $\omega_E(t)=\binom{t+m}{m}$.

The next proposition yields a number $s_0$, depending only on $m$ and the set of minimal elements of $E$ with respect to the product order, such that for all $s\geq s_0$ equality \eqref{equalcomp} holds.

\begin{proposition}\label{agriculture}
Let $E\subseteq \NN^m$ and denote by $M$ the set of minimal elements of $E$ with respect to the product order (which is a finite set by Dickson's lemma). Let $D=0$ if $M$ is empty, and otherwise set 
$$\displaystyle D=\sum_{\xi\in M}\ord\xi.$$ 
Then, for all $s\geq m(D-1)$, we have $\omega_E(s)=|V_E(s)|$. 
\end{proposition}
\begin{proof}
One could prove this using the arguments in the proof of \cite[Proposition~2.2.11]{KLMP}, but we prefer to give a more direct argument. We proceed by induction on $m$ and $D$. The cases $m=1$ or $D=0$ are obvious. So now assume $m>1$ and $D>0$. The result clearly holds when $M$ is empty (equivalently, $E$ is empty). So we may assume that there is $\zeta=(v_1,\dots,v_m)\in M$. Moreover, since $D>0$, $\zeta$ is not the zero tuple; without loss of generality, we assume that $v_m\neq 0$.

Consider the following sets
$$E_1=\{(u_1,\dots,u_{m-1})\in \NN^{m-1}: (u_1,\dots,u_{m-1},0)\geq \xi \text{ for some }\xi\in E\}$$
and
$$E_2=\{(u_1,\dots,u_m)\in \NN^m: (u_1,\dots,u_m+1)\geq \xi \text{ for some }\xi\in E\}.$$
Letting $M_i$ be the minimal elements of $E_i$ for $i=1,2$, we see that $\displaystyle \sum_{\xi\in M_1}\ord\xi\leq D$. On the other hand, the tuple $\eta=(v_1,\dots,v_m-1)$ is in $M_2$, as $\zeta\in M$ and $v_m>0$, and so, since $\ord \eta<\ord \zeta$, we have $\displaystyle \sum_{\xi\in M_2}\ord\xi\leq D-1$. By induction, we have $\omega_{E_1}(s)=|V_{E_1}(s)|$ for all $s$ such that $s\geq (m-1)(D-1)$. Also, $\omega_{E_2}(s-1)=|V_{E_2}(s-1)|$ for all $s$ such that $s-1\geq m(D-2)$. 

A straightforward computation yields
$$|V_E(s)|=|V_{E_1}(s)|+|V_{E_2}(s-1)|, \quad \text{ for all }s,$$
which in turn implies $\omega_E(t)=\omega_{E_1}(t)+\omega_{E_2}(t-1)$ (this is because equality holds for sufficiently large $s$). Now let $s\geq m(D-1)$. Putting the above equalities together, we get
$$
\omega_{E}(s)= |V_{E}(s)|
$$
 as desired.
\end{proof}

Letting $x=(x_1,\dots,x_n)$ be an $n$-tuple of differential variables, we recall that the \emph{canonical orderly ranking} $\unlhd$ on the set $\{\d^\xi x_i:\xi\in \NN^m, i=1,\dots,n\}$
is defined as: $\d_1^{u_1}\cdots\d_m^{u_m}x_i\unlhd\d_1^{v_1}\cdots\d_m^{v_m}x_j$ if and only if
\begin{equation}\label{canrank}
(\sum_k u_k,i,u_1,\dots,u_m)\leq_{\text{lex}} (\sum_k v_k,j,v_1,\dots,v_m)
\end{equation}
where $\leq_{\text{lex}}$ denotes the (left) lexicographic order. The \emph{leader} of a differential polynomial $f\in k\{x\}\setminus k$ is the highest $\d^\xi x_i$ that appears in $f$ with respect to $\unlhd$, and the order of $f$ is the order of its leader. Given any collection of differential polynomials $\Sigma \subset k\{x\}\setminus k$, by a leader of $\Sigma$ we mean a leader of one of its elements and by the order of $\Sigma$ we mean the maximum order among its elements. To avoid certain technicalities that are unnecessary for our purposes, we will not give the precise definition of a characteristic set. Let us just say that a \emph{characteristic set} of a prime differential ideal $P\subset k\{x\}$ is a finite subset of $P$ which is ``reduced'' and ``minimal'' with respect to the canonical orderly ranking $\unlhd$. We refer the reader to \cite[Chapter I]{KolchinDAAG} for further details. 

We can now state the following result of Kolchin's. 
\begin{fact}\cite[Chapter II, \S12]{KolchinDAAG}\label{use1}
Let $P$ be a prime differential ideal of $k\{x\}$ and $\Lambda$ a characteristic set of $P$. If for each $1\leq i\leq n$ we denote by $E_i$ the set of all $\xi\in\NN^m$ such that $\d^\xi x_i$ is a leader of $\Lambda$, then
$$\omega_{P}(t)=\sum_{i=1}^n \omega_{E_i}(t).$$
In particular, the Kolchin polynomial $\omega_{P}$ is completely determined by the set of leaders of any characteristic set of $P$.
\end{fact}

We will make use of a recent upper bound for the order of a characteristic set obtained in \cite{GLS}. Let $A:\NN\times\NN\to \NN$ be the Ackermann function. From this (nonprimitive recursive) function, we build $C_{r,m}^n$, for $r\geq 0$ and $n,m\geq 1$, as follows:
$$C_{0,m}^1=0, \quad\; C_{r,m}^1=A(m-1,C_{r-1,m}^1), \quad \text{ and } \quad C_{r,m}^n=C_{C_{r,m}^{n-1},m}^1.$$
For example, a straightforward computation yields
$$C_{r,1}^n=r, \quad\; C_{r,2}^n=2^n r \quad \text{ and }\quad C_{r,3}^1=3(2^r-1).$$

From \cite[Proposition 6.1]{GLS} we have
\begin{fact}\label{boundchar}
Let $\Sigma\subset K\{x_1,\dots,x_n\}$ be of order at most $r$ and $P$ any of its (minimal) prime components. Then a characteristic set for $P$ has order at most $C_{r,m}^n$.
\end{fact}

We can now prove the main result of this section.

\begin{theorem}\label{sbound}
Let $V\subseteq K^n$ be a differential variety over $k$ (not necessarily irreducible) given by differential polynomials of order at most $r$. Set 
$$s_0:=m\,C_{r,m}^n\binom{C_{r,m}^n+m-1}{C_{r,m}^n}-m.$$ 
Then, for all $s\geq s_0$ and any irreducible component $W$ of $V$, we have
$$\omega_W(s)=\operatorname{trdeg}_k k(\d^{\xi}b_i:\ord\xi\leq s, i=1,\dots,n)$$
where $b=(b_1,\dots,b_n)$ is a generic point of $W$.
\end{theorem}
\begin{proof}
%Let $W$ be an irreducible component of $V$ of maximal Kolchin polynomial. Then, by definition, $\omega_V=\omega_W$. 
Let $\Lambda$ be a characteristic set of the prime differential ideal $P$ in $k\{x_1,\dots,x_n\}$ corresponding to $W$. By Fact \ref{use1}, 
$$\omega_{W}(t)=\sum_{i=1}^n \omega_{E_i}(t).$$
where $E_i$ is the set of all $\xi\in\NN^m$ such that $\d^\xi x_i$ is a leader of $\Lambda$. By Fact~\ref{boundchar}, the elements of $\Lambda$  have order at most $C_{r,m}^n$. Hence, the elements of the set $M_i$ of minimal elements of $E_i$ (with respect to the product order) have order at most $C_{r,m}^n$. Since the number of $m$-tuples of order $s\in\NN$ is $\binom{s+m-1}{s}$, we get that the number of elements in $M_i$ is at most 
$$\displaystyle \binom{C_{r,m}^n+m-1}{C_{r,m}^n},$$
and so
$$\sum_{\xi\in M_i}\ord \xi\leq C_{r,m}^n \binom{C_{r,m}^n+m-1}{C_{r,m}^n}.$$

Suppose $s\geq s_0$. Proposition~\ref{agriculture} now yields
$\omega_W(s)=\sum_{i=1}^n|V_{E_i}(s)|$. Let $b=(b_1,\dots,b_n)$ be a generic point of $W$. 
%\textbf{( I meant to replace the sentence in parenthesis by the following sentence}: Note that if $\xi$ is above or equal an element of $E_i$, then $\d^\xi b_i$ is algebraic \textbf{transcendental} over \begin{equation}\label{use2}k(\d^\zeta b_j: \d^\zeta x_j<\d^\zeta x_i). \end{equation}On the other hand, if $\xi$ is not above or equal an element of $E_i$, then $\d^\xi b_i$ must be algebraic  over \eqref{use2}. Otherwise there would be $f\in P$ with leader $\d^\xi x_i$, contradicting the fact that $\Lambda$ is a characteristic set of $P$.)
Note that the set 
$$N_s=\{\delta^{\xi}b_i: \xi\in V_{E_i}(s),\, 1\leq i\leq n\}$$ 
is algebraically independent over $k$ (because $\Lambda$ is a characteristic set of $P$). Also, if $\xi$ is above or equal to an element of $E_i$, then $\d^\xi b_i$ is algebraic over 
$$k(\d^\zeta b_j: \d^\zeta x_j<\d^\xi x_i);$$ 
and so, by induction and transitivity of algebraic field extensions, $\d^\xi b_i$ is algebraic over $N_s$. We thus have that $$\sum_{i=1}^n|V_{E_i}(s)|=\operatorname{trdeg}_k k(\d^{\xi}b_i:\ord\xi\leq s, i=1,\dots,n),$$
from which the desired equality follows.
\end{proof}

\begin{remark}
Note that in the ordinary case (i.e., $m=1$) the value of $s_0$ in Theorem~\ref{sbound} reduces to $r-1$. 
This  special case of the bound was obtained in \cite[Theorem~12]{DGMS} for quasi-regular ordinary differential systems. 
To the authors' knowledge, up until now there was no such bound established for the partial  differential case.
\end{remark}

We will also need the following bound that witnesses eventual domination of the Kolchin polynomials of the components of a differential variety. 

\begin{proposition}\label{domination}
Let $V\subseteq K^n$ be a differential variety over $k$ (not necessarily irreducible) given by differential polynomials of order at most $r$. Set 
$$s_1:=n\,2^{m+1}\, m!\, D^m+1$$
where 
$$D=C_{r,m}^n\binom{C_{r,m}^n+m-1}{C_{r,m}^n}.$$ 
Suppose $W_1$ and $W_2$ are components of $V$. Then, $\omega_{W_1}> \omega_{W_2}$, with respect to eventual domination, if and only if, for all $s> s_1$, we have $\omega_{W_1}(s)> \omega_{W_2}(s)$.
\end{proposition}
\begin{proof}
Write the Kolchin polynomial of $W_i$ in (standard) numerical form; that is,
$$\omega_{W_i}=\sum_{j=0}^m a_{i,j}\binom{t+j}{j}$$
where the $a_{i,j}$'s are integers. By \cite[Corollary 3.3]{LS}, we have
\begin{equation}\label{usefor}
|a_{i,j}|\leq n\, D^m \quad \text{ for }j=0,\dots,m.
\end{equation}
It is easy to show, by induction say, that if one writes $(t+1)\cdots(t+m)=\sum_{j=0}^m c_jt^j$, then $c_j\leq 2^{m-1}m!$. From this it easy to show, by induction and using \eqref{usefor}, that if we write $\omega_{W_i}$ as $\sum_{j=0}^m b_{i,j}x^j$, then $|b_{i,j}|\leq n2^mm! D^m$. Also note that, because the $a_{i,j}$'s  are integers, all the products $b_{i,j}\cdot m!$ are integers.

Using the above observations and Cauchy's bound on polynomial roots (in terms of the coefficients of the given polynomial), we get that any root of the difference $\omega_{W_1}-\omega_{W_2}$ must be bounded (in absolute value) by $n\,2^{m+1}\,m!\, D^m+1$, that is, by $s_1$. Thus, if $\omega_{W_1}$ eventually dominates $\omega_{W_2}$, then, for $s>s_1$, we must have $\omega_{W_1}(s)> \omega_{W_2}(s)$.
\end{proof}

\section{Definability of Kolchin polynomials} \label{effective}

In this section we prove the main result of the paper, namely that having a prescribed Kolchin polynomial is a definable property in families of differential varieties. For our proof, we will make use of prolongation spaces. We use the notation fixed in previous sections. 

Given an $n$-tuple $b$ from our universal differential field $(K,\D)$ and $s$ a nonnegative integer, we let $\nabla _s (b)$ be the tuple in $K^{n\cdot\binom{m+s}{m}}$ consisting of $b$ and its derivatives of order at most $s$. The ordering of the tuple is not particularly important, but for convention, we will order the tuple with respect to the canonical orderly ranking as in \eqref{canrank}. 

\begin{defn} 
Given a differential variety $V \subseteq K^n$ and a nonnegative integer $s$, the prolongation of $V$ is defined as 
$$B_s (V) := \nabla _s (V)^{\operatorname{Z-cl}} \subseteq K^{n \cdot \binom{m+s}{m}}$$ 
where $*^{\operatorname{Z-cl}}$ denotes Zariski-closure. 
\end{defn}

In Proposition~\ref{Definv1} below we prove that if $(V_a)$ is a definable family of differential varieties, then, for each $s$, the family $(B_s(V_a))$ is also definable. Our proof uses the following fact, which follows from results in \cite{FLS}.

\begin{fact}\label{degbound} Recall that $*^{\operatorname{Z-cl}}$ denotes Zariski-closure.
\begin{enumerate}
\item If $V\subseteq K^n$ is a differential variety given by differential polynomials of degree $d$ and order $R$, then there is $D=D(d,R,m,n)$, which can be effectively computed, such that $\deg V^{\operatorname{Z-cl}}\leq D$.
\item Given a definable family $(V_a)$ of differential varieties, the family $((V_a)^{\operatorname{Z-cl}})$ is definable as well. Moreover, one can effectively compute a formula defining this family.
\end{enumerate}
\end{fact}
\begin{proof}
(1) This is precisely the content of \cite[Corollary 4.5 and Remark~4.7(2)]{FLS}. Note that an explicit formula is provided there that computes $D(d,R,m,n)$.

(2) Suppose the family $(V_a)$ has degree $d$ and order $R$ (as defined in Section~\ref{intro}). This means that, for any $a$, the differential variety $V_a$ is given by  differential polynomials of degree $d$ and order $R$. Also, assume each $V_a$  is in $K^n$. By (1), $\deg (V_a)^{\operatorname{Z-cl}}\leq D$, where $D$ only depends on the data $(d,R,m,n)$. So, by \cite[Proposition 3]{Heintz} and Kronecker's theorem (see \cite[Chapter 7,\S17]{Ritt}), we can find $(n +1)$-many polynomials $g_0,...,g_{n}$ of degree at most $D$ such that $(V_a)^{\operatorname{Z-cl}}$ is defined by $g_0=\cdots=g_{n}=0$.  

A collection of polynomials defining $(V_a)^{\operatorname{Z-cl}}$ uniformly can be obtained by fixing $(n+1)$-many generic polynomials in $z$ of the form $G_u(z)=\{g_0(z,u), ..., g_{n}(z,u)\}$ of degree $D$ and noting that the set of points $(a,b)$ such that the $g_i(z,b)$'s define $(V_a)^{\operatorname{Z-cl}}$ is the same as the set of points $(a,b)$ such that the points of $V_a$ are solutions to the $g_i(z,b)$'s and there is no other Zariski-closed set defined by $(n+1)$-many polynomials of degree at most $D$ with this property contained in the solution set of the $g_i(z,b)$'s. The latter condition is easily expressible by a formula in the language of differential rings.
\end{proof}

Part (2) of the above fact is referred to as Zariski-closure being definable in families of differential varieties. As a result we can prove

\begin{prop} \label{Definv1} 
Let $(V_a)$ be a definable family of differential varieties with each $V_a$ in $K^n$, and let $s$ a nonnegative integer. Then $(B_s (V_a))$ has the structure of a definable family, and one can effectively compute a formula defining this family.
\end{prop} 
\begin{proof}
The differential variety $\nabla_s(V_a)$ is given by the equations of $V_a$ together with
$$z=\nabla_s(x)$$
where $x$ are variables for $K^n$ and $z$ are variables for $K^{n\alpha_s}$ that coincide with $x$ in the first $n$ coordinates where $\alpha_s:=\binom{m+s}{m}$. Thus, the family of differential varieties $(\nabla_s(V_a))$ is definable. By part (2) of Fact~\ref{degbound}, the family $((\nabla_s(V_a))^{\operatorname{Z-cl}})$ is definable, and one can effectively derive a formula defining it. Since $B_s(V_a)=(\nabla_s(V_a))^{\operatorname{Z-cl}}$, we are done.
\end{proof}

We will make use of the following well known algebro-geometric fact (for various elementary proofs see \cite{Hrushovski92} and \cite[Introduction and Appendix]{EDC}): 

\begin{fact} \label{factdef} Let $(X_a)$ be a definable family of algebraic varieties (in the language of rings). Then, for fixed $d$ and $\ell$, the set 
$$\{a:\dim X_a\geq d \text{ and $X_a$ has at least $\ell$-many top-dimensional components} \}$$
is definable by a formula in the language of rings. Moreover, a formula defining this set can be effectively computed.
\end{fact}

\begin{corollary}\label{dimdef}
Let $(V_a)$ be a definable family of differential varieties, and $d, \ell, s$ nonnegative integers. Then the set
$$\{a \, : \, \dim B_s (V_a) \geq d \text{ and $B_s(V_a)$ has at least $\ell$-many top-dimensional components}\}$$ 
is definable, and a formula can be effectively computed.
\end{corollary}
\begin{proof}
Copying the notation in the proof of Proposition~\ref{Definv1}, we have a generic family of polynomials $G_z(x)$ such that for specific $(a,b)$ there is a formula (in the language of differential rings) such that $G_b(x)$ defines $B_s(V_a)$. If we add to this the formula defining 
$$\{b\,:\,G_b(x) \text{ has dimension $\geq d$ and at least $\ell$-many top-dimensional components}\},$$ 
which exists by Fact~\ref{factdef}, we obtain the desired formula.
\end{proof}

We can now prove the (effective) definability of Kolchin polynomials in families:

\begin{thm} \label{Boundy}
Let $(V_a)$ be a definable family of differential varieties and let $p$ be a numerical polynomial. Then the sets 
\begin{equation}\label{ineq}
\{a \, : \, \omega _{V_a }  \geq  p \}
\end{equation}
and
\begin{equation}\label{eq}
\{a \, : \, \omega _{V_a } =  p \}
\end{equation}
are definable in the structure $(K,\D)$. 
\end{thm} 

\begin{proof} 
First note that, by Fact~\ref{kolpol}(ii), if $p$ has degree larger than $m$ then both sets \eqref{ineq} and \eqref{eq} are empty. So we assume that $\deg(p)\leq m$. Now note that \eqref{eq} follows from \eqref{ineq}. Indeed, by Remark~\ref{wellorder}(2), there is a minimum Kolchin polynomial $q$ such that $q>p$, but then \eqref{eq} equals
$$\{a  \, : \, \omega _{V_a }  \geq  p\text{ and } \omega _{V_a }  <  q\}$$
which is definable by \eqref{ineq}.

We now prove \eqref{ineq}. Since the coefficients of any polynomial (in one variable over $\mathbb Q$) of degree at most $m$ are determined by its values at $m+1$ distinct natural numbers (using invertibility of Vandermonde matrices for instance), if we can effectively bound (within given intervals) the values of $\omega_{V_a}$ at $m+1$ distinct natural numbers then we can effectively decide when $\omega_{V_a}\geq p$. This is because a numerical polynomial $q$ dominates $p$ iff the standard coefficients (as a numerical polynomial) of $q$ are greater than or equal to those of $p$ in the lexicographical order.

Suppose the family $(V_a)$ has degree $d$ and order $r$. Let $s_1$ be as in Proposition~\ref{domination} (note that $s_1\geq s_0$ with $s_0$ defined as in Theorem~\ref{sbound}). Let $W$ be an irreducible component of $V_a$ (for fixed but arbitrary $a$) of maximal Kolchin polynomial. By Theorem~\ref{sbound} and Proposition~\ref{domination}, for all $s> s_1$, we have that $\dim B_s(W)\geq \dim B_s(W')$ for $W'$ any other component of $V_a$. Since $B_s(V_a)$ equals the union of the prolongation of its components, for $s> s_1$ we get that $\dim B_s(V_a)=\dim B_s(W)$. The upshot is that now Theorem~\ref{sbound} yields $\omega_{V_a}(s)=\dim B_s(V_a)$ for all $a$ and $s> s_1$. 

Now the result follows from Corollary \ref{dimdef}, as it shows that we can effectively find a formula that determines those $a$'s such that $\dim B_s(V_a)$ is within a fixed interval, and hence the same applies to $\omega_{V_a}(s)$ for any $s> s_1$, as desired.
\end{proof} 

\begin{rem} 
The various steps (in the current and previous section) for the proof of Theorem \ref{Boundy} are effective in the sense that they provide a general recipe for a specific formula giving the collection of fibres with some fixed Kolchin polynomial. 
\end{rem} 

%Given a differential variety $V$, we say that a component $W$ of $V$ has {\em maximal Kolchin polynomial} if $\omega_{W}=\omega_V$. 

From the proof of Theorem~\ref{Boundy}, we can deduce the following:

\begin{proposition}\label{onmaxpol}
Let $(V_a)$ be a definable family of differential varieties and $\ell$ a nonnegative integer. Then the set
\begin{equation}\label{manycomp}
\{a:\, V_a \text{ has at least $\ell$-many components of maximal Kolchin polynomial}\}
\end{equation}
is definable in the structure $(K,\D)$.
\end{proposition}
\begin{proof}
Using the notation of the proof of Theorem~\ref{Boundy}, we see that for $s> s_1$, we have $\omega_{V_a}(s)=\omega_{V_i}(s)$ for each $i$ such that $B_s(V_i)$ is top-dimensional in $B_s(V)$. Thus, for such $V_i$'s we get $\omega_{V_a}=\omega_{V_i}$; in other words, such $V_i$'s are the components of $V$ of maximal Kolchin polynomial. This shows that the set \eqref{manycomp} equals the set of those $a$'s such that $B_s(V_a)$ has at least $\ell$-many top-dimensional components for $s=s_1,\dots, s_1+m$. But the latter set of $a$'s is definable by Corollary~\ref{dimdef}.
\end{proof}

%(\textbf{This is not true! A counter-example: Let $m=1$ and the differential variety $V=\mathbb V(y'^2-4y)\subset\mathbb A^1$. Then $V$ has two components $V_1=\mathbb V(\sat(y'^2-4y))$ and $V_2=\mathbb V(y)$. $B_0(V)=\mathbb A^1$ has only one component $\mathbb A^1$, $B_1(V)=\mathcal V(y'^2-4y)\subset\mathbb A^2$ has only one component $B_1(V_1)$. Note that $B_1(V_2)$ is embedded in $B_1(V_1)$. Even the following is invalid: For fixed $s\geq s_0$, if $V_1,\dots,V_k$ are all the irreducible components of $V_a$ of maximal Kolchin polynomial, then $B_s(V_1),\ldots, B_s(V_k)$ are precisely the components of  $B_s(V_a)$ of top-dimension. Non-example: in the ordinary case, let $V=\mathbb V(y_1)\cup \mathbb V(y''_1,y''_2)\subset\mathbb A^2$ and $\mathbb V(y_1)$ is clearly the dominant component of $V$. Take $s_0=2$. But for $s=s_0$, the dominant component of $B_s(V)$ is $B_s(\mathbb V(y''_1,y''_2))$. For $s=s_0+1$, the dominant component of $B_s(V)$ is $B_s(\mathbb V(y''_1,y''_2))$ and $B_s(\mathbb V(y_1))$. That is because Kolchin polynomials of components are ordered by eventual domination. A bound for the eventual domination of Kolchin polynomials should be larger than the bound of order of characteristic sets. For ordinary case, $ns_0$ is such a bound, but for the partial case, I am not sure.})

\section{Some applications of definability} \label{applications}

Using Theorem~\ref{Boundy} and Proposition~\ref{onmaxpol}, we can prove the definability of other interesting differential-algebraic properties. For instance,

\begin{corollary}
Let $(V_a)$ be a definable family of differential varieties, $p$ a numerical polynomial, and $\ell$ a nonnegative integer. Then the set
$$\{a:\, \omega_{V_a}=p \text{ and $V_a$ has at least $\ell$-many components with Kolchin polynomial $p$}\} $$
is definable in the structure $(K,\D)$.
\end{corollary}
\begin{proof}
The above set is simply the intersection of the sets
$$\{a:\, \omega_{V_a}=p\} $$
and 
$$\{a:\, \text{ and $V_a$ has at least $\ell$-many components of maximal Kolchin polynomial}\} $$
The former is definable by Theorem~\ref{Boundy}, while the latter is definable by Proposition~\ref{onmaxpol}.
\end{proof}

Recall that a differential variety $V$ is said to be {\em weakly irreducible} if it has exactly one component of maximal Kolchin polynomial. Proposition \ref{onmaxpol} has the following immediate consequence.

\begin{corr}
Given a definable family $(V_a)$ of differential varieties, the set
$$\{a:\, V_a \text{ is weakly irreducible}\}$$
is definable in the structure $(K,\D)$.
\end{corr}

Note that being weakly irreducible is not equivalent to being (fully) irreducible. In fact the question of definability of irreducibility for differential varieties is remarkably difficult (and remains open); it is actually equivalent to the generalized Ritt problem which is a longstanding problem since the 1950's (see \cite[Theorem 1]{GKO}).

We recall that the differential type of $V$, usually denoted by $\tau_V$, is defined as the degree of the Kolchin polynomial of $V$. We now show that the differential type is also a property that is definable in families.

\begin{corollary}
Let $(V_a)$ be a definable family of differential varieties, and $d$ a nonnegative integer. Then, the set
$$\{a :\, \tau_{V_a}\geq d\}$$
is definable.
\end{corollary}
\begin{proof}
Consider the set of all Kolchin polynomials with differential type at least $d$. By Remark~\ref{wellorder}(2), this set has a smallest element (with respect to eventual domination), say $p$. By Theorem~\ref{Boundy}, the set $\{a:\, \omega_{V_a}\geq p\}$ is definable. The latter is precisely the set of those $a$'s such that the differential type of $V_a$ is at least $d$.
\end{proof}

\smallskip

Our last result says that a definable family admits only finitely many Kolchin polynomials.

\begin{corollary} 
Let $(V_a)$ be a definable family of differential varieties. Then the set 
$$\{\omega_{V_a}:\, \text{ as $V_a$ varies in the family}\}$$ 
is finite.
\end{corollary} 
\begin{proof} 
Towards a contradiction assume the set is infinite. Let $\mathfrak N$ be the collection of numerical polynomials of degree at most $m$. For each $p\in \mathfrak N$, consider the set
$$A_p=\{a: \omega_{V_a}\neq p\}$$
The set $A_p$ is definable by Theorem \ref{Boundy}. Moreover, by our assumption, the intersection of finitely many of the $A_p$'s is nonempty. But by universality of $K$ (or saturation rather), the intersection $\displaystyle \bigcap_{p\in \mathfrak N}A_p$ would contain a point, say $a_0$. Of course, this is impossible since $\omega_{V_{a_0}}$ is a numerical polynomial of degree at most $m$.
\end{proof}

\bibliographystyle{plain}

\end{document}